\documentclass[12pt]{amsart}
\usepackage[english]{babel} 
\usepackage{amsfonts} 
\usepackage{amsmath} 
\usepackage{amsthm} 
\usepackage{amssymb} 
\usepackage{mathrsfs}
\usepackage{mathtools}
\usepackage{tikz-cd}
\usepackage{rotating}
\usepackage{verbatim}

\title[On the Howe correspondence...]{On the Howe correspondence and Harish-Chandra series}
\author{Jesua Epequin Chavez}
\date{}

\DeclareMathOperator{\End}{End}
\DeclareMathOperator{\Hom}{Hom}
\DeclareMathOperator{\U}{U}
\DeclareMathOperator{\Or}{O}
\DeclareMathOperator{\GL}{GL}
\DeclareMathOperator{\Sp}{Sp}
\DeclareMathOperator{\Irr}{Irr}
\DeclareMathOperator{\Ind}{Ind}

\DeclareMathOperator{\SO}{SO}
\DeclareMathOperator{\sgn}{sgn}

\DeclareMathOperator{\St}{St}

\DeclareMathOperator{\diag}{diag}

\DeclareMathOperator{\rank}{rank}

\DeclareFontEncoding{LS1}{}{}
\DeclareFontSubstitution{LS1}{stix}{m}{n}
\DeclareSymbolFont{symbols2}{LS1}{stixfrak} {m} {n}
\DeclareMathSymbol{\operp}{\mathbin}{symbols2}{"A8}

\tikzset{labr/.style={anchor=north, rotate=90, inner sep=1mm}}
\tikzset{labl/.style={anchor=south, rotate=90, inner sep=.5mm}}

\newtheorem{thm}{Theorem}

\newtheorem{defi}{Definition}

\newtheorem{lem}{Lemma}
\newtheorem{prop}{Proposition}
\newtheorem{cor}{Corollary}

\newtheorem*{thm*}{Theorem}

\begin{document}
\begin{abstract}
We study the effect of the Howe correspondence on Harish-Chandra series for type I dual pairs $(\U_m(\mathbb{F}_q),\U_n(\mathbb{F}_q))$ and $(\Sp_{2m}(\mathbb{F}_q),\Or^\pm_{2n}(\mathbb{F}_q))$, where $\mathbb{F}_q$ denotes the finite field with $q$ elements ($q$ odd). We use the Lusztig correspondence to reduce the study of the Howe correspondence between Harish-Chandra series to a correspondence between unipotent Harish-Chandra series. Finally, we provide two applications of these results.
\end{abstract}
\maketitle
\section*{Introduction}
Let $\mathbb{F}_q$ be a finite field with $q$ elements and odd characteristic. A pair of reductive subgroups of $\Sp_{2n}(q)$, where each one is the centralizer of the other, is called \emph{reductive dual pair}. We focus our attention on \emph{irreducible} dual pairs (cf. \cite{Kudla}). One such pair can be \emph{linear} $(\GL_m(q),\GL_{m'}(q))$, \emph{unitary} $(\U_m(q),\U_{m'}(q))$ or \emph{symplectic-orthogonal} $(\Sp_{2m}(q),\Or_{m'}(q))$, with $n=mm'$ in all cases. The last two are also called \emph{type I} dual pairs.

For a reductive dual pair $(G_m,G'_{m'})$, Roger Howe defined the a correspondence $\Theta_{m,m'} : \mathscr{R}(G_m)\rightarrow\mathscr{R}(G'_{m'})$ between the categories of complex representations. This correspondence is known as the \emph{Howe correspondance}, and arises from the restriction to $G_m\times G'_{m'}$ of the \emph{Weil representation} $\omega$ of the symplectic group $\Sp_{2n}(q)$ (cf. \cite{Howe}). 

Let $\mathbf{G}$ be a connected group defined over $\mathbb{F}_q$, and $\mathbf{G}^*$ its dual group. Denote by $G$ and $G^*$ their groups of rational points. In \cite{Lusztig3} Lusztig defined a partition of the set $\Irr(G)$ of irreducible representations of $G$ into \emph{Lusztig series} $\mathscr{E}(G,(s))$. These series are parametrized by rational conjugacy classes of semisimple elements $s$ of $G^*$. The elements of $\mathscr{E}(G,(1))$ are called \emph{unipotent representations} of $G$.

In general, the Howe correspondence is not compatible with unipotent representations. Therefore, we make use of a similar correspondance $\Theta^\flat_{m,m'}:\mathscr{R}(G_m)\rightarrow\mathscr{R}(G'_{m'})$ arising from a Weil representation $\omega^\flat$ introduced by G\'erardin in \cite{Gerardin}. This correspondence preserves unipotent representations (cf. Proposition 2.3 in \cite{AMR}).

Let $G_m$ belong to a type I dual pair. \emph{Harish-Chandra series} for this group are indexed by cuspidal pairs $(\GL_\textbf{t}\times G_{m-|\textbf{t}|},\boldsymbol{\sigma}\otimes\varphi)$, where $\mathbf{t}=(t_1,\ldots,t_r)$ is a partition of norm $|\mathbf{t}|\leq m$, $\GL_\textbf{t}=\GL_{t_1}\times\cdots\times\GL_{t_r}$, and $\boldsymbol{\sigma}$ (resp. $\varphi$) is an irreducible cuspidal representations of $\GL_\textbf{t}$ (resp. $G_{m-|\textbf{t}|}$). Let $\varphi'$ (resp. $l'$) denote the first occurrence (resp. first occurrence index) of $\varphi$ (cf. \cite{AM}). Our first main result generalizes Theorem 3.7 in \cite{AMR} (see Theorem \ref{HoweHarish-Chandra}). 

\begin{thm*}
The Howe correspondence $\Theta_{m,m'}^\flat$ sends $\Irr(G_m,\boldsymbol{\sigma}\otimes\varphi)$ to $\mathscr{R}(G'_{m'},\boldsymbol{\sigma}'\otimes\varphi')$ whenever $m'\geq l'$ and to zero otherwise. In the first case, $\boldsymbol{\sigma}'=\boldsymbol{\sigma}\otimes 1$ if $|\mathbf{t}'|\geq |\mathbf{t}|$, and $\boldsymbol{\sigma}=\boldsymbol{\sigma}'\otimes 1$ otherwise. 
\end{thm*}

From now on we restrict our attention to unitary dual pairs. Let $\rho$ and $\rho'$ be the cuspidal representations in the previous theorem, so that $\Theta^\flat_{m,m'}$ sends the Harish-Chandra series $\Irr(G_m,\rho)$ of $G_m$ to the series $\Irr(G'_{m'},\rho')$ of $G'_{m'}$. The second goal of this paper is to express this correspondence between Harish-Chandra series in terms of a correspondence between unipotent Harish-Chandra series. The latter was studied by of Aubert, Michel and Rouquier in \cite{AMR}, they reduced it to a combinatorial problem. 

 In order to achieve this goal we make use of the \emph{Lusztig correspondence}. It is a bijection between the series $\mathscr{E}(G,(s))$ and the series $\mathscr{E}(C_{G^*}(s),(1))$, sending a representation $\rho$ of $G$ to a unipotent representation $\rho_u$ of the centralizer $C_{G^*}(s)$. For classical groups, this centralizer can be expressed as a product of smaller reductive groups. For instance, when $G_m$ is a unitary group 
$$
C_{G_m}(s) \simeq G_\#(s)\times G_{m-l}
$$
where $G_\#(s)$ is a product of linear or unitary groups whose rank is equal to $l$. This decomposition allows us to identify the series $\mathscr{E}(G_m,(s))$ with the direct product of $\mathscr{E}(G_\#(s),(1))$ and $\mathscr{E}(G_{m-l},(1))$. We denote by $\pi_\#\otimes\pi_{m-l}$ the representation correponding to $\pi$ under this bijection.



Denote by $\omega^\flat_{m,m',\rho}$ the restriction of $\omega^\flat_{m,m'}$ to $\mathscr{R}(G_m,\rho)\otimes\mathscr{R}(G'_{m'},\rho')$, by $\mathbf{R}_G$ the natural representation of $G\times G$ on $\mathscr{L}(G)$ (see Section \ref{PreliminariesRepresentationTheory}), and let $\mathbf{R}_{G,\rho}$ denote its projection onto $\mathscr{R}(G,\rho)\otimes\mathscr{R}(G,\rho)$.  Our second main result is the following (see Theorem \ref{HC-UnipotentHC}).

\begin{thm*}
The representation $\omega^\flat_{m,m',\rho}$ is identified with $\mathbf{R}_{G_\#(s),\rho_\#}\otimes \omega^\flat_{m-l,m'-l,\rho_{m-l}}$  via the bijection
\begin{align*}
\Irr(G_m\times G'_{m'},\rho\otimes\rho')\simeq \Irr(C_{G_m}(s)\times C_{G'_{m'}}(s'),\rho_u\otimes \rho'_u),
\end{align*}
where $s$ and $s'$ are rational semisimple elements of $L$ and $L'$ whose Lusztig series contain $\rho$ and $\rho'$ respectively.
\end{thm*}

This theorem reduces the study of the Howe correspondence to the study of the correspondence between unipotent Harish-Chandra series. Indeed, since $\rho$ is a cuspidal representation, $\rho_{m-l}$ is a cuspidal \emph{and} unipotent representation. Thanks to recent work by Pan (cf. \cite{Pan2}) it should be possible to extend the previous theorem to symplectic-orthogonal pairs. We finish the paper with two applications. 

We provide a generalisation of results in a previous paper. For $\pi$ in $\Irr(G_m)$ (not necessarily unipotent) we define an order on $\Theta^\flat_{m,m'}(\pi)$ which is compatible with the one introduced in \cite{Epequin}, and find minimal and maximal representations on this set (see Theorem \ref{ExtremalRepresentations}). 

Using a result by Howlett and Lehrer (cf. \cite{HL}) we identify $\omega^\flat_{m,m',\rho}$ to a representation $\Omega_{m,m',\rho}$ of $W_{\mathbf{G}_m}(\rho)\times W_{\mathbf{G}'_{m'}}(\rho')$ (see Section \ref{SectionExtremalRepresentations}). We use the previous theorem to express $\omega^\flat_{m,m',\rho}$ in terms of $\mathbf{R}_{W_{G_\#(s)}(\rho_\#)}$ and the representation $\Omega_{m-l,m'-l,\rho_{m-l}}$ (see Theorem \ref{GeneralizedWeylGroupCorrespondence}). As mentioned before, the latter was made explicit in \cite{AMR} (cf. Theorem 3.10 in that paper).

\section{Howe correspondance and cuspidal support}\label{HoweCorrespondanceCuspidalSupport}
Theorem 3.7 of \cite{AMR} states that the correspondence $\Theta^\flat$ is compatible with unipotent Harish-Chandra series . In this section, we generalize this theorem to arbitrary Harish-Chandra series. The proof follows that of Theorem 2.5 in \cite{Kudla}, and relies on the Schr\"odinger mixed model together with some basic results concerning induced representations of finite groups.
 
\subsection{Some preliminary results}\label{PreliminariesRepresentationTheory}
Through all the statements $G$ and $N$ denote arbitrary finite groups and $H$ denotes a subgroup of $G$. The proofs are straightforward computations.

Let $G$ act on a set $X$ and $(\varphi,V)$ be a representation of $G$. Take $x\in X$ and denote by $G\cdot x$ its orbit, by $\St(x)$ its stabilizer, and by $\mathscr{S}(Z, V)$ the vector space of functions defined on $Z \subset X$ with values in $V$.  
\begin{lem}\label{ind-from-stab}
The representation $\varphi$ and the action of $G$ on $X$ induce a  linear action on $\mathscr{S}(G\cdot x, V)$. Moreover, we have a $G$-isomorphism,
$$
\mathscr{S}(G\cdot x, V) \simeq \Ind_{\St(x)}^G \varphi.
$$
\end{lem}
\begin{lem}\label{semi-rep}
Let $G$ act on $N$ and $\varphi$ be a representation of $G\ltimes N$. Then,
$$
\varphi(g)\varphi(n)\varphi(g^{-1})=\varphi(g\cdot n).
$$
\end{lem}
In other words, giving a representation of a semidirect product $G\ltimes N$ amounts to give two representations, $\omega$ of $G$ and $\rho$ of $N$, such that 
$$
\omega(g)\rho(n)\omega(g^{-1}) = \rho(g\cdot n).
$$
\begin{lem}\label{semi-ind}
Let $\varphi$ be a representation of $H\ltimes N$. 
\begin{itemize}
\item[a)] The formula $\phi(f)(g,n)=\varphi(g(n))f(g)$ defines an equivalence of representations, 
$$
\phi:\Ind_H^G(\varphi\rvert_H)  \simeq (\Ind_{H\ltimes N}^{G\ltimes N}\varphi)|_G.
$$
\item[b)] Restriction induces a $G$-isomorphism,
$$
(\Ind_{H\ltimes N}^{G\ltimes N}\varphi)_N\simeq \Ind_H^G (\varphi)_N. 
$$
\end{itemize}  
\end{lem}
\begin{lem}\label{char-tens-ind}
Let $\chi$ be a representation of $G$, and $\varphi$ a representation of $H$. We have a $G$-isomorphism
$$
\chi\otimes\Ind_H^G \varphi\simeq \Ind_H^G (\chi\rvert_H\otimes\varphi)
$$
\end{lem}
Denote by $\mathbf{R}_G$, the natural representation of the group $G\times G$ on the space $\mathscr{S}(G)$, of functions defined on $G$ with values in $\mathbb{C}$. This representation is isomorphic to the one obtained by inducing the trivial representation of $G$ to $G\times G$ (diagonal inclusion). It decomposes as 
$$
\mathbf{R}_G = \sum_{\pi\in\Irr(G)} \pi\otimes \hat{\pi},
$$
where $\hat{\pi}$ denotes the contragredient representation of $\pi$. Denote by $\iota$ the involution of $\Irr(G)$ sending $\pi$ to $\hat{\pi}$.

The following result concerns the $\mathbb{F}_q$-rank of certain rational Levi subgroups (see Definition 8.3 in \cite{Digne-Michel}). In its statement $\epsilon_\mathbf{G}$ denotes $(-1)^{\mathbb{F}_q-\rank}$. We provide the proof for lack of reference.
\begin{prop}\label{RationalLeviFqRank}
If $\mathbf{M}$ is a rational Levi contained in a rational parabolic subgroup of $\mathbf{G}$, then $\epsilon_\mathbf{M}$ is equal to $\epsilon_\mathbf{G}$. 
\end{prop}
\begin{proof}
By definition, we need to prove the equality between the $\mathbb{F}_q$-rank of $\mathbf{M}$ and $\mathbf{G}$. Let $\mathbf{P}$ be the rational parabolic containing $\mathbf{M}$, and consider a rational maximal torus $\mathbf{T}$ of $\mathbf{M}$. We can choose a rational Borel subgroup $\mathbf{B}$ of $\mathbf{G}$ contained in $\mathbf{P}$ and containing $\mathbf{T}$. Since it is also contained in the rational Borel subgroup $\mathbf{B}\cap\mathbf{M}$ of $\mathbf{M}$, the  $\mathbb{F}_q$-rank of $\mathbf{M}$ and $\mathbf{G}$ are equal.
\end{proof}

\subsection{Mixed Schr\"odinger model}
Denote by $W_n$ a symplectic space of dimension $2n$ and fix a linear representation $\psi$ of $\mathbb{F}_q$. Denote by $\rho_n$ the \emph{Heisenberg representation} of $H(W_n)$, the Heisenberg group; and by $\omega_n$ the \emph{Weil representation} (cf. \cite[\S 1.2]{Epequin}) of $\Sp(W_n)$, the symplectic group. They verify the intertwining relation
$$
\omega_n(x)\rho_n(h)\omega_n(x^{-1}) = \rho_n (x \cdot h).
$$
This, according to Lemma \ref{semi-rep}, is the same as giving a representation of the semidirect product $\Sp(W_n)\ltimes H(W_n)$ whose restrictions to $\Sp(W_n)$ and $H(W_n)$ are $\omega_n$ and $\rho_n$ respectively.

If $W_n = V_1 \operp V_2$ is an orthogonal sum of symplectic vector spaces, then there is a short exact sequence
$$
1\rightarrow \mathbb{F}_q \xrightarrow{i} H(V_1)\times H(V_2) \xrightarrow{j} H(W_n) \rightarrow 1,
$$
with arrows defined by $i(t)=(t,-t)$ and $j((v_1,t_1),(v_2,t_2))=(v_1+v_2,t_1+t_2)$. Let $(\rho_1,S_1)$ and $(\rho_2,S_2)$ be models of the Heisenberg representations of $H(V_1)$ and $H(V_2)$ respectively, then $(\rho_1\otimes\rho_2,S_1\otimes S_2)$ is a model for the Heisenberg representation of $H(W_n)$. Moreover, the metaplectic representation $\omega_n$ of $\Sp(W_n)$ coincides with $\omega_1\otimes\omega_2$ on its subgroup $\Sp(V_1)\times\Sp(V_2)$.

 Let $\{e_1,\ldots,e_n,e'_n,\ldots,e'_1\}$ be a symplectic base of $W_n$. Let $k\leq n$ and $X_k$ (resp. $Y_k$) be a totally isotropic subspace spanned by the $k$ first (resp. last) vectors in this base. The non-degenerate pairing $X_k\times Y_k\rightarrow \mathbb{C}$ induced by the symplectic form, allows to identify $Y_k$ with the dual space $\check{X}_k$ of $X_k$. Therefore, there is an \emph{Witt decomposition} :
$$
W_n\simeq (X_k\oplus \check{X}_k)\operp W_{n-k}.
$$ 
The Heisenberg representation of $H(X_k\oplus \check{X}_k)$ can be realized in the space $\mathscr{S}(\check{X}_k)$. Let $(\rho_{n-k},S_{n-k})$ be a model for the Heisenberg representation of the group $H(W_{n-k})$. From the previous paragraph, the tensor product $\mathscr{S}(\check{X}_k)\otimes S_{n-k}$ provides a model for the Heisenberg representation of $H(W_n)$, called \emph{mixed Schr\"odinger model}. It can be identified to $\mathscr{S}(\check{X}_k,S_{n-k})$. 
\begin{lem}\label{mix-Heisenberg}\cite[Exemple I.4]{VMW}
The action of the Heisenberg representation $\rho_n$ in $\mathscr{S}(\check{X}_k,S_{n-k})$ is given by,
$$
\rho_n(w,t)f(\check{y}) = \psi(\langle\check{y},x\rangle+\langle\check{x},x\rangle/2+t)\rho_{n-k}(w_{n-k})f(\check{x}+\check{y}),
$$
where $w=x+w_{n-k}+\check{x}$, $x\in X_k$, $w_{n-k}\in W_{n-k}$ and $\check{x},\check{y}\in \check{X}_k$.
\end{lem}
Explicit formulas for the action of the Weil representation $\omega_n$ on the vector space $\mathscr{S}(\check{X}_k,S_{n-k})$ are also known. Let $P_k$ be the stabilizer of $X_k$ in $\Sp(W_n)$. It is a maximal parabolic group with a Levi decomposition $P_k=M_k\ltimes N_k$. Its standard Levi subgroup $M_k$ consist of matrices $m(a,u) = \diag(a,u,{}^\mathrm{t}a^{-1})$ for $a\in\GL_k$ and $u\in\Sp(W_{n-k})$. The unipotent radical $N_k$ is the group of matrices
$$
n(c,d) = \left(\begin{matrix}
1 & c & d-c {}^\mathrm{t} c/2\\
   & 1 & -{}^\mathrm{t}c\\
   &    & 1
\end{matrix}\right),
$$
where $d$ is a symmetric matrix. 

Consider the representation $\chi$ of $P_k\ltimes H(X_k\oplus W_{n-k})$, trivial on $N_k$, defined by
$$
\chi(m) = \det(a)^{(q-1)/2}\omega_{n-k}(u). 
$$
on elements $m=m(a,u)$ of $M_k$, and by
$$
\chi(x+w_0,t)=\rho_{n-k}(w_0,t)
$$
on elements of the Heisenberg group. There is a  unique representation of $\Sp(W_n)\ltimes H(W_n)$ whose restriction to $P_k\ltimes H(W_n)$ agrees with the induced representation of $\chi$ to this group, \cite[Theorem 2.4]{Gerardin}. The space $\mathscr{S}(\check{X}_k,S_{n-k})$ can be identified with the vector space of this induced representation, on which $P_k$ acts by right translations. 
\begin{lem}\label{mix-meta}
The action of $P_k$ on $\mathscr{S}(\check{X}_k,S_{n-k})$ is given by
$$
\omega_n(m)f(\check{x})  = \det(a)^{(q-1)/2}\omega_{n-k}(u)f(\check{a}\check{x}) 
$$
for elements $m=m(a,u)$ of Levi subgroup $M_k$, and by
$$
\omega_n(n) f(\check{x})  = \psi(\langle d\check{x},\check{x}\rangle/2)\rho_{n-k}(\check{c}\check{x})f(\check{x}),
$$
for elements $n=n(c,d)$ of the unipotent radical $N_k$.
\end{lem}

\subsection{Weil representation coinvariants}
From now on we will fix two Witt towers, $\mathscr{T}$ and $\mathscr{T}'$, such that $(G_m,G'_{m'})$ is a type I dual pair for any $G_m\in\mathscr{T}$ and $G'_{m'}\in\mathscr{T}'$. 

Let $D$ be a field equal to $\mathbb{F}_q$ when the dual pair is symplectic-orthogonal, and equal to $\mathbb{F}_{q^2}$ when the pair is unitary. Denote by $N$ the norm of $D$ over $\mathbb{F}_q$. Let $W_m$  (resp. $W'_{m'}$) be the $D$-vector space of Witt index $m$ (resp. $m'$) on which $G_m$ (resp. $G'_{m'}$) acts; denote by $n$ (resp. $n'$) its dimension. 

Let $X_k$ denote the totally isotropic subspace of $W_m$, spanned by the $k$ first vectors of a hermitian base, $k\leq m$. Let $P_k$ be the stabilizer of $X_k$ in $G_m$. Denote by $N_k$ its unipotent radical, $\GL_k=\GL(X_k)$ and $M_k=\GL_k\times G_{m-k}$ the standard Levi subgroup of $P_k$. Let $Q_j$ be the stabilizer of $X_j$ in $\GL_k$, $j \leq k$. It is a parabolic group whose elements are matrices $\left(\begin{smallmatrix}a_j & \ast \\  & a_{k-j} \end{smallmatrix}\right)$. Finally, denote by $X'_{k'}$, $\GL'_{k'}$, $P'_{k'}$, $N'_{k'}$, $M'_{k'}$ and $Q'_{j'}$ the analogous for $G'_{m'}$.

\begin{thm}\label{coinv-filtr} Let $\prescript{*}{}{R}_k$ and $\prescript{*}{}{R'}_{k'}$  be the parabolic restriction functor from $G_m\times G'_{m'}$ to $M_k\times G'_{m'}$ and $G_m\times M'_{k'}$ respectively.

\emph{a)} There exists a $M_k\times G'_{m'}$-invariant filtration $0=\tau_0\subset \tau_1\subset \cdots \subset \tau_{r+1}=\prescript{*}{}{R}_k(\omega_{m,m'}^\flat)$,  where $r=\min\{k,m'\}$. Its successive quotients $\tau_{i+1}/\tau_i$ verify
 $$
 \tau_{i+1}/\tau_i \simeq \Ind_{Q_{k-i}G_{m-k}\times P'_i}^{M_k\times G'_{m'}} \mathbf{R}_{\GL_{i}}\otimes\omega_{m-k,m'-i}^\flat.
 $$
 
\emph{b)} Likewise, there is a $G_m\times M'_{k'}$-invariant filtration $0=\tau'_0\subset \tau'_1\subset \cdots \subset \tau'_{r'+1}=\prescript{*}{}{R'}_{k'}(\omega_{m,m'}^\flat)$,  where $r'=\min\{k',m\}$. Its successive quotients $\tau'_{i+1}/\tau'_i$ verify
 $$
 \tau'_{i+1}/\tau'_i \simeq \Ind_{P_i\times Q'_{k'-i}G'_{m'-k'}}^{G_m\times M'_{k'}}\mathbf{R}_{\GL_{i}}\otimes\omega_{m-i,m'-k'}^\flat.
 $$
\end{thm}
The proof of the previous theorem is long, so it is presented in the next section. An easy consequence of it is the following.
\begin{cor}\label{coinv-submod}
\begin{itemize}
\item[a)] The parabolic restriction $\prescript{*}{}{R}_k(\omega_{m,m'}^\flat)$ verifies :
$$
\prescript{*}{}{R}_k(\omega_{m,m'}^\flat)\simeq\bigoplus_{i=0}^{\min\{k,m'\}} \Ind_{Q_{k-i}G_{m-k}\times P'_i}^{M_k\times G'_{m'}} \mathbf{R}_{\GL_{i}}\otimes\omega_{m-k,m'-i}^\flat.
$$
\item[b)] Likewise, the parabolic restriction $\prescript{*}{}{R}'_{k'}(\omega_{m,m'}^\flat)$ verifies
$$
\prescript{*}{}{R}'_{k'}(\omega_{m,m'}^\flat)\simeq \bigoplus_{i=0}^{\min\{k',m\}} \Ind_{P_i\times Q'_{k'-i}G'_{m'-k'}}^{G_m\times M'_{k'}}\mathbf{R}_{\GL_{i}}\otimes\omega_{m-i,m'-k'}^\flat.
$$
\end{itemize}
\end{cor}

\subsection{Proof of Theorem \ref{coinv-filtr}}
We start by calculating the coinvariants relative to the group $N_{k,1}$ consisting of matrices $n(0,d)$. Then, we compute the coinvariants of the group $N_{k,2}$ consisting of matrices $n(c,0)$. Due to the short exact sequence
$$
1\rightarrow N_{k,1}\rightarrow N_k \rightarrow N_{k,2} \rightarrow 1,
$$
performing these two computations is equivalent to calculating the coinvariants of $N_k$. In each calculation, we use a mixed Schr\"odinger model because it allows us to express the Jacquet functor in terms of restriction of functions.  

\textit{Coinvariants relative to $N_{k,1}$} : Consider the decomposition $W_m  \simeq X_k \oplus W_{m-k} \oplus \check{X}_k$, where $\check{X}_k$ denotes the dual space of $X_k$. Tensoring by $W'_{m'}$ produces
$$
W_m \otimes W'_{m'} \simeq ( X_k \otimes W'_{m'} \oplus  \check{X}_k \otimes W'_{m'} ) \operp W_{m-k} \otimes W'_{m'} 
$$   

The mixed Sch\"odinger model corresponding to this orthogonal sum makes $\Sp(W_m\otimes W'_{m'})$ act on $\mathscr{S}(\check{X}_k\otimes W'_{m'},S_{m-k,m'})$ where $S_{m-k,m'}$ is a model of the metaplectic representation of $H(W_{m-k} \otimes W'_{m'})$. Lemma \ref{mix-meta} implies the action of the parabolic subgroup $P_k \times G'_{m'}$ of $G_m \times G'_{m'}$ on $\mathscr{S}(\check{X}_k\otimes W'_{m'},S_{m-k,m'})$ is given by
$$
\omega_{m,m'}^\flat(m,g')f(\check{x}) = \omega^\flat_{m-k,m'}(u,g')f(g'\check{x}a^{-1})
$$
if $m = m(a,u)$ belongs to $M_k$, and $g'$ belongs to $G'_{m'}$. Also, 
$$
\omega^\flat_{m,m'}(n)f(\check{x}) = \psi(\langle d \check{x},\check{x}\rangle/2)\rho_{m-k,m'}(\check{c}\check{x})f(\check{x}),
$$
if $n=n(c,d)$ belongs to $N_k$, $f$ belongs to $\mathscr{S}(\check{X}_k\otimes W'_{m'},S_{m-k,m'})$, $\check{x}$ belongs to $\check{X}_k \otimes W'_{m'} \simeq \Hom(X_k,W'_{m'})$, and $\langle\cdot,\cdot\rangle$ is the natural paring between $\check{X}_k \otimes W'_{m'}$ and its dual, .
 
Therefore, the action of $n(0,d) \in N_{k,1}$ is given by
$$
\omega_{m,m'}^\flat(n(0,d))f(\check{x}) = \psi(\langle d \check{x},\check{x}\rangle/2)f(\check{x}),
$$
so the $f\in \mathscr{S}(\check{X}_k\otimes W'_{m'},S_{m-k,m'})$ fixed by $N_{k,1}$ are those whose support is contained in $\mathscr{S}(Z,S_{m-k,m'})$ where $Z$ is the subspace consisting of those $x \in \Hom(X_k,W'_{m'})$ whose image is totally isotropic. This implies the following.
\begin{prop}
The restriction to $Z$ defines a $P_k\times G'_{m'}$-isomorphism between the space of coinvariants relative to $N_{k,1}$ and $\mathscr{S}(Z,S_{m-k,m'})$. 
\end{prop} 
The subspace $Z$ is invariant under the action of $P_k\times G'_{m'}$. The orbits of this action are $Z_i = \{\check{x}\in Z\rvert \rank{\check{x}}=i\}$ for $i=0,\ldots,\min\{k,m'\}$. We denote $\min\{k,m'\}$ by $r$ for short.
 
This orbit decomposition $Z=\cup Z_i$ provides the following finite $P_k\times G'_{m'}$-invariant filtration
$$
\{0\}=\mathscr{S}_0 \subset \mathscr{S}_1 \subset \cdots \subset \mathscr{S}_{r+1}= \mathscr{S}(Z,S_{m-k,m'}), 
$$ 
where $\mathscr{S}_i= \{f\in \mathscr{S}(Z,S_{m-k,m'})\rvert  f(\check{x})=0 \mbox{ for } \check{x}\mbox{ with }\rank(\check{x})\geq i\}$. The subquotients of this filtration are $\mathscr{S}_{i+1}/\mathscr{S}_i \simeq \mathscr{S}(Z_i,S_{m-k,m'})$, for  $i=0,\ldots,r$. 
 
For each $i$ take $z_i\in Z_i$ and consider the representation $(\varphi_i,S_{m-k,m'})$ of $P_k\times G'_{m'}$ defined by 
\begin{align*}
\varphi_i(mn,g') = \omega^\flat_{m-k,m'}(u,g')\rho_{m-k,m'}(\check{c} z_i) 
\end{align*}
for $g'\in G'_{m'}$, $n=n(c,d) \in N_k$ and $m=m(a,u)\in M_k$. The action of $M_k\times G'_{m'}$ via this representation and the natural action of $M_k\times G'_{m'}$ on $Z_i$ induce a representation of $M_k\times G'_{m'}$ on $\mathscr{S}(Z_i,S_{m-k,m'})$. This representation is equal to the restriction of $\mathscr{S}_{i+1}/\mathscr{S}_i$ to $M_k\times G'_{m'}$. Hence, Lemma \ref{ind-from-stab} provides a $M_k\times G'_{m'}$-isomorphism,
$$
\mathscr{S}_{i+1}/\mathscr{S}_i \simeq \Ind_{\St(z_i)  G_{m-k}}^{M_k\times G'_{m'}}\varphi_i.
$$
where $\St(z_i)$ is the stabilizer of $z_i$ in $\GL_k\times G'_{m'}$.

The group $P_k\times G'_{m'}$ can be written as a semidirect product $M_k\times G'_{m'} \ltimes N_k$. The stabilizer $H_i$ of $z_i$ inside of $P_k\times G'_{m'}$ can be also expressed as a semidirect product $H_i=\St(z_i) G_{m-k}\ltimes N_k$, so that Lemma \ref{semi-ind} provides an isomorphism of $M_k\times G'_{m'}$-modules
$$
\phi : \mathscr{S}_{i+1}/\mathscr{S}_i \simeq \Ind_{H_i}^{P_k\times G'_{m'}}\varphi_i.
$$
That lemma also gives an explicit formula for $\phi$ which in this case becomes
$$
\phi(f)(mn,g') = \varphi_i(mn,g')f(g'z_ia^{-1}),
$$
where $m=m(a,u)\in M_k$, $g'\in G'_{m'}$ and $n\in N_k$. A direct calculation shows that $\phi$ is \`a fortiori a $N_k$-morphism and hence an isomorphism of $P_k\times G'_{m'}$-modules.

Finally, Lemma \ref{semi-ind} gives us a $M_k\times G'_{m'}$-isomorphism 
\begin{align}\label{suc-quot}
(\mathscr{S}_{i+1}/\mathscr{S}_i)_{N_k} \simeq \Ind_{\St(z_i)  G_{m-k}}^{M_k\times G'_{m'}}(\varphi_i)_{N_k}.
\end{align} 
To continue the proof we need to compute the coinvariants of $\varphi_i$ relative to $N_k$. Due to the fact that $N_{k,1}$ already acts trivially, this computation comes down to that of the coinvariants relative to $N_{k,2}$.  
 
\textit{Coinvariants relative to $N_{k,2}$} : Let $X_i'$ be the image of $z_i$. Consider the decomposition $W'_{m'}\simeq X'_i\oplus W'_{m'-i}\oplus \check{X}'_i$. Tensoring by $W_{m-k}$ we obtain
\begin{align}\label{mix-mod-2}
W_{m-k} \otimes W'_{m'} \simeq ( X'_i \otimes W_{m-k} \oplus  \check{X}'_i \otimes W_{m-k} ) \operp W_{m-k} \otimes W'_{m'-i}. 
\end{align}

The related mixed model yields a realisation of $\varphi_i$ in $S_{m-k,m'}\simeq\mathscr{S}(X'_i\otimes W_{m-k},S_{m-k,m'-i})$, where $S_{m-k,m'-i}$ is a model of the metaplectic representation of $H(W_{m-k}\otimes W'_{m'-i})$. The explicit action of $N_k$ on this space is obtained thanks to Lemma \ref{mix-Heisenberg} :
\begin{align*}
\varphi_i(n)f(\check{x}) &= \rho_{m-k,m'}(\check{c}z_i)f(\check{x}) = \psi(\langle \check{c}z_i,\check{x}\rangle)f(\check{x}),
\end{align*}
where $n=n(c,d)\in N_k$ and $\check{x}\in X'_i\otimes W_{m-k}$.
 
We conclude that the $f \in \mathscr{S}(X'_i\otimes W_{m-k},S_{m-k,m'-i})$ invariant by $N_k$ are those trivial on $X'_i\otimes W_{m-k}\setminus\{0\}$. This provides the following
 
\begin{prop}
The restriction to zero provides an isomorphism of $M_k\times G'_{m'}$-modules between $\mathscr{S}(0,S_{m-k,m'-i})\simeq S_{m-k,m'-i}$ and the space $(S_{m-k,m'})_{N_k}$ of coinvariants of $\varphi_i$ relative to $N_k$.
\end{prop}

It is easy to see that if $(a,g')$ belongs $\St(z_i)$ then $g'$ stabilizes $X'_i$. Hence, $\St(z_i)G_{m-k}$ is contained in $M_k\times P'_i$, where $P'_i$ is the stabilizer of $X'_i\subset W'_{m'}$. The mixed model formulas for the orthogonal sum decomposition (\ref{mix-mod-2}) allow us to compute the action of $G_{m-k}\times P'_i$ by $\omega^\flat_{m-k,m'}$ on $\mathscr{S}(W_{m-k}\otimes X'_i,S_{m-k,m'-i})$. From this we deduce the action of $\St(z_i)  G_{m-k} \subset M_k\times P'_i$ by the representation $\varphi_i$ on $S_{m-k,m'-i}$ is
\begin{align}\label{varphi-act}
\varphi_i(m,m'n')=\omega^\flat_{m-k,m'-i}(u,u').
\end{align}
for $m=m(a,u)\in M_k$, $m'=m'(a',u')\in M'_i$ and $n'\in N'_i$. We remark that the action of the unipotent radical $N'_i$ is trivial.

To finish our computation we need to make explicit the elements of the stabilizer $\St(z_i)$ of $z_i$ in $\GL_k\times G'_{m'}$. If $(a,g')\in \St(z_i)$, then $a$ preserves the kernel $X_{k-i}$ of $z_i$. Hence, $\St(z_i)$ is contained in the parabolic $Q_{k-i}\times P'_i$ of $\GL_k\times G'_{m'}$, where $Q_{k-i}$ is the stabilizer of $X_{k-i}$ in $X_k$. Moreover, for the given basis $z_i$ equals $\left(\begin{smallmatrix}0_{k-i,i}&1_{i,i}\\0&0\end{smallmatrix}\right)$, so $(a,g')$ belongs to $\St(z_i)$, if and only if 
$$
a = \left(\begin{matrix} a_{k-i} & \ast\\ & a_i \end{matrix}\right) \mbox{, }g' = \left(\begin{matrix} a_i & \ast & \ast\\ & u' & \ast \\  & & \check{a}_i^{-1}\end{matrix}\right). 
$$
Considering $a\in Q_{k-i}$ and $g'=m'n'\in P'_i$ (where $m'=m'(a_i,u')$) as in the previous line, the action (\ref{varphi-act}) becomes
$$
\varphi_i(m,m'n')=\omega^\flat_{m-k,m'-i}(u,u'),
$$ 
denoting $(\mathscr{S}_i)_{N_k}$ by $\tau_i$, the isomorphism (\ref{suc-quot}) turns into
$$
\tau_{i+1}/\tau_i \simeq \Ind_{\St(z_i)  G_{m-k}}^{M_k\times G'_{m'}} 1 \otimes\omega^\flat_{m-k,m'-i}.
$$
Using transitivity, we need to compute the induction from $\St(z_i)  G_{m-k}$ to $Q_{k-i}G_{m-k}\times P'_i$. This comes down to induce the trivial representation from $\GL_i$ to $\GL_i\times\GL_i$ (the inclusion being diagonal). Part (a) follows then from Lemma \ref{char-tens-ind}. Part (b) has an analogous proof.  

\subsection{The main theorem}
Let $G_m$ be a symplectic, orthogonal or unitary group in Witt tower $\mathscr{T}$, let $T$ be its torus of diagonal matrices, and $B$ the Borel of upper triangular matrices. The set of standard Levi subgroups of $G_m$ (with respect to $T$ and $B$) can be parametrized by sequences $\mathbf{t}=(t_1,\ldots,t_r)$, such that $|\mathbf{t}|= \sum_{i=1}^r t_i$ is not greater than $m$. The corresponding Levi subgroup $L_{\mathbf{t}}$ is equal to $\GL_{t_1}\times \cdots \times \GL_{t_r} \times G_{m-|\mathbf{t}|}$. For this Levi, we denote parabolic induction by  $R_{\mathbf{t}}$ and parabolic restriction by $\prescript{*}{}{R_{\mathbf{t}}}$.
 
Let $\pi$ be an irreducible representation of $G_m$. There exists a set $\mathbf{t}(\pi)=(t_1,\ldots,t_r)$ of positive integers with $|\mathbf{t}(\pi)|\leq m$, and cuspidal irreducible representations $\sigma_i$ of $\GL_{t_i}$ and $\varphi$ of $G_{m-|\mathbf{t}(\pi)|}$ such that 
$$
[\pi] = [\sigma_1,\ldots,\sigma_r,\varphi]
$$ 
is the cuspidal support of $\pi$. 

Following Section 2 in \cite{AM}, when $\pi$ is a cuspidal representation of $G_m$, we denote by $\theta^\flat(\pi)$ (resp. $m'(\pi)$) its first occurrence (resp. first occurrence index).
 
\begin{thm}\label{HoweCuspidalSupport}
Let $\pi$ be an irreducible representation of $G_m$ with cuspidal support $[\sigma_1,\ldots,\sigma_r,\varphi]$. Let $\pi'$ belong to $\Theta_{m,m'}^\flat(\pi)$, and denote $\mathbf{t}(\pi)$ and $\mathbf{t}(\pi')$ by $\mathbf{t}$ and $\mathbf{t}'$ respectively.
 
\emph{a)} If $|\mathbf{t}'|\geq |\mathbf{t}|$ then
$$
[\pi']=[\sigma_1,\ldots,\sigma_r,1,\ldots,1,\theta^\flat(\varphi)].
$$
 
\emph{b)} When $|\mathbf{t}'| < |\mathbf{t}|$ there exists a sequence $i_1\leq\cdots\leq i_d$, with $d=|\mathbf{t}| - |\mathbf{t}'|$, such that $\sigma_{i_k}=1$ and
$$
[\pi']=[\sigma_1,\ldots, \widehat{\sigma_{i_1}},\ldots,\widehat{\sigma_{i_d}},\ldots,\sigma_r,\theta^\flat(\varphi)].
$$
\end{thm}
\begin{proof}
First, take $\pi$ cuspidal and assume that $m'>m'(\pi)$. Let $[\pi']=[\sigma'_1,\ldots,\sigma'_{r'},\varphi']$ and let $\sigma'$ be an irreducible representation of $\GL_{|\mathbf{t}'|}$ with cuspidal support $[\sigma']=[\sigma'_1,\ldots,\sigma'_{r'}]$ such that $\pi'$ is an irreducible component of $R'_{|\mathbf{t}'|}(\sigma'\otimes\varphi')$.

Since
$$
\langle\omega^\flat_{m,m'},\pi\otimes\pi'\rangle \leq \langle\omega^\flat_{m,m'},\pi\otimes R'_{|\mathbf{t}'|}(\sigma'\otimes\varphi')\rangle,
$$ 
Frobenius reciprocity and Corollary \ref{coinv-submod} imply that
$$
\langle\omega^\flat_{m,m'},\pi\otimes\pi'\rangle \leq \sum_{i=0}^{\min\{m,|\mathbf{t}'|\}}\langle \tau'_{i+1}/\tau'_i,\pi\otimes\sigma'\otimes\varphi'\rangle.
$$
Due to the fact that $\pi$ is cuspidal, only the term corresponding to $i=0$ can contribute, and that term gives  
$$
\langle \omega^\flat_{m,m'-|\mathbf{t}'|},\pi\otimes\sigma'\otimes\varphi'\rangle=\langle 1,\sigma'\rangle\langle \omega^\flat_{m,m'-|\mathbf{t}'|},\pi\otimes\varphi'\rangle,
$$ 
so in this case we have $\sigma'=1$, and $\varphi'$ is the first occurrence $\theta^\flat(\pi)$ of $\pi$. Therefore
$$
[\pi']=[1,\ldots,1,\theta^\flat(\pi)].
$$
This proves part (a) in the case when $r=0$, i.e. when $\pi$ is cuspidal.

Now, suppose $\pi$ is non-cuspidal with support $[\pi]=[\sigma_1,\ldots,\sigma_r,\varphi]$. There is an irreducible representation $\varphi_1$ of $G_{m-t_1}$ with cuspidal support $[\varphi_1]=[\sigma_2,\ldots,\sigma_r,\varphi]$ such that $\pi$ is an irreducible component of $R_{t_1}(\sigma_1\otimes\varphi_1)$. Once again, Frobenius reciprocity and Corollary \ref{coinv-submod} imply that
$$
\langle\omega^\flat_{m,m'},\pi\otimes\pi'\rangle \leq  \sum_{i=0}^{\min\{t_1,m'\}}\langle \tau_{i+1}/\tau_i,\sigma_1\otimes\varphi_1\otimes\pi'\rangle.
$$ 
Since $\sigma_1$ is cuspidal only the terms corresponding to $i=0$ and $i=t_1$ contribute to the sum.

\emph{I}. The term corresponding to $i=0$ is 
$$
\langle 1\otimes\omega^\flat_{m-t_1,m'},\sigma_1\otimes\varphi_1\otimes \pi'\rangle=\langle 1,\sigma_1\rangle\langle\omega^\flat_{m-t_1,m'},\varphi_1\otimes \pi'\rangle.
$$
This is non-trivial if and only if $\sigma_1=1$ and $\pi'$ belongs to $\Theta^\flat_{m-t_1,m'}(\varphi_1)$. The former equality implies that the trivial character of $\GL_{t_1}$ is cuspidal so that $t_1=1$.
 
\emph{II}. The term corresponding to $i=t_1$ yields 
$$
\langle\mathbf{R}_{\GL_{t_1}}\otimes\omega^\flat_{m-t_1,m'-t_1},\sigma_1\otimes\varphi_1\otimes \prescript{*}{}{R'}_{t_1}(\pi')\rangle.
$$
This is non-zero only if there is a simple $M'_{t_1}$-submodule $\sigma'_1\otimes\varphi'_1$ of $\prescript{*}{}{R'}_{t_1}(\pi')$ such that 
\begin{align*}
\langle\mathbf{R}_{\GL_{t_1}}& \otimes\omega^\flat_{m-t_1,m'-t_1},
 \sigma_1\otimes\varphi_1\otimes\sigma'_1\otimes\varphi'_1\rangle \neq 0.
\end{align*}
This in turn equals  $\langle \sigma_1,\sigma'_1\rangle\langle\omega^\flat_{m-t_1,m'-t_1},\varphi_1\otimes\varphi'_1\rangle$. Hence, in this case we obtain $\sigma'_1=\sigma_1$ and $\varphi'_1\in\Theta^\flat_{m-t_1,m'-t_1}(\varphi_1)$.

Suppose that $|\mathbf{t}'| \geq |\mathbf{t}|$. If the first case above holds, since $|\mathbf{t}'| \geq |\mathbf{t}\setminus \mathrm{t}_1|$, we can apply an inductive argument to $\pi'\in\Theta^\flat_{m-t_1,m'}(\varphi_1)$ and $\varphi_1$ in order to obtain
$$
[\pi']=[\sigma_2,\ldots,\sigma_r,1,\ldots,1,\theta^\flat(\varphi)].
$$ 
Performing a permutation yields (a), because in this case $\sigma_1=1$.   If the second case above holds then $t'_1=t_1$ and $\sigma'_1=\sigma_1$. Since $|\mathbf{t}'\setminus \mathrm{t}_1| \geq |\mathbf{t}\setminus \mathrm{t}_1|$, the inductive hypothesis applied to $\varphi'_1$ and $\varphi_1$ yields (a).

Now suppose that $|\mathbf{t}'| < |\mathbf{t}|$.  In this situation the first case above must occur at least $d=|\mathbf{t}| - |\mathbf{t}'|$ times. This would yield a subsequence $\sigma_{i_i},\ldots,\sigma_{i_d}$ is as in part (b) of the Theorem \ref{HoweCuspidalSupport}. 
\end{proof}  

This theorem can be restated in terms of Harish-Chandra series. Let $\boldsymbol{\sigma}$ denote the cuspidal representation $\sigma_1\otimes\ldots\otimes\sigma_r$ of $\GL_{\mathbf{t}}=\GL_{t_1}\times\ldots\GL_{t_r}$, and let $l'$ be the first occurrence index of $\varphi$.

\begin{thm}\label{HoweHarish-Chandra}
The Howe correspondence $\Theta_{m,m'}^\flat$ sends $\Irr(G_m,\boldsymbol{\sigma}\otimes\varphi)$ to $\mathscr{R}(G'_{m'},\boldsymbol{\sigma}'\otimes\theta^\flat(\varphi))$ whenever $m'\geq l'$ and to zero otherwise. In the first case, $\boldsymbol{\sigma}'=\boldsymbol{\sigma}\otimes 1$ if $|\mathbf{t}'|\geq |\mathbf{t}|$, and $\boldsymbol{\sigma}=\boldsymbol{\sigma}'\otimes 1$ otherwise. 
\end{thm}
\begin{proof}
Suppose there is a representation $\pi$ in $\Irr(G_m,\boldsymbol{\sigma}\otimes\varphi)$ such that $\Theta^\flat_{m,m'}(\pi)$ is non-zero, and let $\pi'$ belong to it. According to Theorem \ref{HoweCuspidalSupport}, the representation $\pi'$ belongs to a Harish-Chandra series $\Irr(G'_{m'},\boldsymbol{\sigma}'\otimes\varphi')$, where $\varphi'=\theta^\flat(\varphi)$. In this case $m'-l'=|\mathbf{t}'|\geq 0$, whence $m'\geq l'$. The rest of the statement follows also from Theorem \ref{HoweCuspidalSupport}.
\end{proof}
If we ask for the representation $\boldsymbol{\sigma}\otimes\varphi$ of $\GL_{\mathbf{t}}\times G_l$ to be also unipotent, then $\mathbf{t}$ becomes $\mathbf{t}=(1^{m-l})$, the representation $\boldsymbol{\sigma}$ becomes the trivial representation of the torus $T_{m-l}$ of diagonal matrices, and $\boldsymbol{\sigma}'$ becomes the trivial representation of the torus $T_{m'-l'}$. Therefore, as a particular case of Theorem \ref{HoweHarish-Chandra} we get Theorem 3.7 in \cite{AMR}.

\begin{cor}\label{HoweUnipotentHC}
The Howe correspondence $\Theta^\flat_{m,m'}$ sends $\Irr(G_m,1\otimes\varphi)$ to $\mathscr{R}(G'_{m'},1\otimes\theta^\flat(\varphi))$ whenever $m'\geq l'$, and to zero otherwise. Moreover, the representation $\varphi$ is unipotent if and only if the same holds for $\theta^\flat(\varphi)$. 
\end{cor}

\section{Howe correspondence and Lusztig correspondence}\label{HoweLusztigCorrespondence}
The purpose of this section is to see the effect of the Lusztig correspondence on the Howe correspondence for unitary dual pairs, and realize the connection this has to results in the previous section. 

\subsection{Centralizers of rational semisimple elements}\label{CentralizersSection}
Let $\mathbf{G}$ be a group defined over $\mathbb{F}_q$, and  $C_\textbf{G}(x)$ be the centralizer of a rational element $x$ in $\textbf{G}$. Denote by $G$ and $C_G(x)$ their groups of rational elements. 

Assume that $\mathbf{G}$ is also connected. In Proposition 5.1 of \cite{Lusztig}, Lusztig found a bijection
\begin{align}\label{LusztigBijection}
 \mathfrak{L}^{G_m}_s : \mathscr{E}(G,(s)) \simeq \mathscr{E}(C_{G^*}(s),(1)).
\end{align}
where $s$ is a rational semisimple element of $\textbf{G}^*$. In Proposition 1.7 of \cite{AMR}, Aubert, Michel and Rouquier extended this bijection to even orthogonal groups. By linearity we get an isometry between the categories $\mathscr{R}(G,(s))$ and $\mathscr{R}(C_{G^*}(s),(1))$ spanned by the Lusztig series $\mathscr{E}(G,(s))$ and $\mathscr{E}(C_{G^*}(s),(1))$ respectively. 

Following Section 1.B in \cite{AMR}, let $\mathbf{G}$ be a classical group of rank $n$ with Frobenius endomorphism $F$, and $\textbf{T}_n\simeq\overline{\mathbb{F}}_q\times\ldots\times\overline{\mathbb{F}}_q$ be its maximal torus of diagonal matrices. By definition, a rational semisimple element is conjugate to an element $(\lambda_1,\ldots,\lambda_n)$ of $\textbf{T}_n$. Let $\nu_\lambda(s)$ the number of times $\lambda$ appears in this list. There is a decomposition
$$
C_\textbf{G}(s)=\prod\mathbf{G}_{[\lambda]}(s),
$$
where $[\lambda]$ is the orbit of $\lambda$ by the action of the Frobenius $F$, intersected with $\{\lambda_1,\ldots,\lambda_n\}$. Each $\mathbf{G}_{[\lambda]}(s)$ is a reductive quasi-simple group of rank $|[\lambda]|\nu_\lambda(s)$. Moreover, if $\lambda\neq \pm 1$, then $\mathbf{G}_{[\lambda]}(s)$ is a unitary or general linear group (possibly over some finite extension of $\mathbb{F}_q$). Additionally :

\noindent (1) If $\mathbf{G}=\boldsymbol{\GL}_n$ is unitary, then $\mathbf{G}_{[\pm 1]}(s)$ is a unitary group.

\noindent (2) If $\mathbf{G}=\boldsymbol{\SO}_{2n+1}$, then $\mathbf{G}_{[-1]}(s)=\boldsymbol{\Or}_{2\nu_{-1}(s)}$, and $\mathbf{G}_{[1]}(s)=\boldsymbol{\SO}_{2\nu_1(s)+1}$.

\noindent (3) If $\mathbf{G}=\boldsymbol{\Or}_{2n}$, then $\mathbf{G}_{[\pm 1]}(s)=\boldsymbol{\Or}_{2\nu_{\pm 1}(s)}$.

In all cases we see that $\mathbf{G}_{[1]}(s)$ is a group of the same kind as $\mathbf{G}$, but of smaller rank. 

\begin{lem}\label{LeviCentralizer}
Let $\mathbf{L}$ be a Levi complement of the parabolic subgroup $\mathbf{P}$ of $\mathbf{G}$, both groups being rational. If $s$ is a semisimple element of $L$, then $\mathbf{L}_{[\lambda]}(s)$ is a Levi subgroup of the parabolic $\mathbf{P}_{[\lambda]}(s)$ of $\mathbf{G}_{[\lambda]}(s)$. Moreover these groups are also rational.
\end{lem}  

\subsection{Weil representation and Lusztig correspondence}

Suppose that $\mathbf{G}$ is an unitary group, and identify $\mathbf{G}^*$ with $\mathbf{G}$. Let $s$ be a rational semisimple element, $\mathbf{G}_\#(s)$ denote the product of $\mathbf{G}_{[\lambda]}(s)$ for $\lambda\neq 1$, and let $l$ be its rank. Let $G_\#(s)$ and $G_{[1]}(s)$ denote the group of rational elements of  $\mathbf{G}_\#(s)$ and $\mathbf{G}_{[1]}(s)$ respectively. If $\mathbf{G}$ is the group $\mathbf{G}_m$ of Witt index $m$, then the group $G_{[1]}(s)$ can be identified with the unitary group $G_{m-l}$. The decomposition of the centralizer of $s$ (found in the previous section) together with the Lusztig correspondence yield a one-to-one correspondence 
\begin{align}\label{BijectionPan}
\Xi^G_s : \mathscr{E}(G_m,(s))\simeq \mathscr{E}(G_\#(s),(1))\times \mathscr{E}(G_{m-l},(1)).
\end{align}
We denote by $\pi_\#$ and $\pi_{m-l}$ the (unipotent) representations of $G_\#(s)$ and $G_{m-l}$ such that $\pi$ corresponds to $\pi_\#\otimes \pi_{m-l}$ by $\Xi^G_s$. This bijection is easily extended to rational Levi subgroups $\mathbf{L}$ of $\mathbf{G}_m$.

Let $(G_m,G'_{m'})$ be an unitary dual pair. According to Proposition 2.3 in \cite{AMR}, if $s$ is a semisimple element in $G_m$, then there is a semisimple element $s'$ in $G'_{m'}$, such that the Howe correspondence sends $\mathscr{E}(G_m,(s))$ to $\mathscr{R}(G'_{m'},(s'))$. Moreover, in this case $s'=(s,1)$ if $m\leq m'$, and $s=(s',1)$ otherwise. 

\begin{prop}\label{PanGroups}
The groups $(G_m)_\#(s)$ and $(G'_{m'})_\#(s')$ are isomorphic. Hence, if $l$ denotes their common rank then $((G_m)_{[1]}(s),(G'_{m'})_{[1]}(s'))$ can be identified with the unitary dual pair $(G_{m-l},G'_{m'-l})$.
\end{prop}
\begin{proof}
The isomorphism between  $(G_m)_\#(s)$ and $(G'_{m'})_\#(s')$ comes from the explicit decomposition of centralizers given in the Section \ref{CentralizersSection}, and the fact that $s$ and $s'$ have the same eigenvalues different from 1 (with same multiplicities), this comes from the last sentence in the previous paragraph. The second assertion follows from the discussion above.
\end{proof}
Let $s_m$ and $s_{m'}$ be semisimple elements of $G_m$ and $G'_{m'}$ whose Lusztig series are related by the Howe correspondence. From previous remarks we deduce that there is some $l\leq\min(m,m')$, and some $s\in \mathbf{T}_l$ with eigenvalues different from $1$, such that $s_m=(s,1)$, and $s_{m'}=(s,1)$. Let $\omega^\flat_{m,m',s}$ denote the projection of the Weil representation $\omega^\flat_{m,m'}$ onto $\mathscr{R}(G_m,(s_m))\otimes\mathscr{R}(G'_{m'},(s_{m'}))$, and $\mathbf{T}_{l,0}$ the subset of $\mathbf{T}_l$ whose elements have all their eigenvalues different from $1$. Proposition 2.4 in \cite{AMR} asserts that
$$
\omega^\flat_{m,m'}=\bigoplus_{l=0}^{\min(m,m')}\bigoplus_{s\in\mathbf{T}_{l,0}}\omega^\flat_{m,m',s}.
$$

The Weil representation $\omega^\flat_{m,m',s}$ can be described in terms of a correspondence between unipotent characters, defined either by $\mathbf{R}_{G_\#(s),1}$, or by the unipotent projection of the Weil representation of the smaller unitary dual pair $(G_{m-l},G'_{m'-l})$.

\begin{thm}\label{ReductionUnipotentCase}\cite[Th\'eor\`eme 2.6]{AMR}
Let $s$ belong to $\mathbf{T}_{l,0}$. For a linear or unitary pair $(G_m,G'_{m'})$, the representation $\omega^\flat_{m,m',s}$ is the image by the Lusztig correspondance
$$
\mathscr{E}(G_m\times G'_{m'},(s_m)\times(s_{m'})) \simeq \mathscr{E}(C_{G_m}(s_m)\times C_{G_{m'}}(s_{m'}),1),
$$
of the representation
$$
\mathbf{R}_{G_\#(s),1}\otimes \omega^\flat_{m-l,m'-l,1}.
$$
\end{thm}
We express this theorem using a commutative diagram (compare with Theorem 3.10 in \cite{Pan}).

\begin{cor}\label{HoweLusztig}
 The following diagram is commutative :
\begin{center} 
 \begin{tikzcd}
 \mathscr{E}(G_m,(s_m)) \arrow[d, "\Theta^\flat_{m,m'}"] \arrow[r, "\Xi^{G_m}_{s_m}", "\sim" swap] 
                                                &  \mathscr{E}(G_\#(s),(1))\times \mathscr{E}(G_{m-l},(1)) \arrow[d, "\iota\hspace{0.05cm}\otimes\hspace{0.05cm}\Theta^\flat_{m-l,m'-l}"] \\
\mathscr{R}(G'_{m'},(s_{m'})) \arrow[r,  "\Xi^{G'_{m'}}_{s_{m'}}", "\sim" swap] 
 												   & \mathscr{R}(G_\#(s),(1))\otimes \mathscr{R}(G'_{m'-l},(1)).
 \end{tikzcd}
\end{center}.
\end{cor}
\begin{proof}
 Let $\pi$ be in irreducible representation in the Lusztig series of $s_m$ and $\pi'$ an irreducible subrepresentation of $\Theta^\flat_{m,m'}(\pi)$. We must stablish that $\pi'_\#=\hat{\pi}_\#$ and that $\pi'_{m'-l}$ is an irreducible subrepresentation of $\Theta^\flat_{m-l,m'-l}(\pi_{m-l})$.
 
 According to Theorem \ref{ReductionUnipotentCase}, $\pi_\#\otimes\pi'_\#$ is an irreducible component of the representation $\mathbf{R}_{G_\#(s),1}$ of $G_\#(s)\times G_\#(s)$. This implies that $\pi'_\#=\hat{\pi}_\#$ (see Section \ref{PreliminariesRepresentationTheory}). The same theorem tells us that $\pi_{m-l}\otimes\pi'_{m'-l}$ is an irreducible constituent of the (unipotent part) of the Weil representation $\omega^\flat_{m-l,m'-l}$ associated to the dual pair $(G_{m-l},G'_{m'-l})$. Hence, $\pi'_{m'-l}$ is an irreducible subrepresentation of $\Theta^\flat_{m-l,m'-l}(\pi_{m-l})$.
\end{proof}

\subsection{Unipotent Harish-Chandra series}
Corollary \ref{HoweLusztig} shows how the study of the Howe correspondence can be brought to the study of a correspondence between unipotent representations of a smaller dual pair of the same kind. In this section we will use this result to show how to describe the Howe correspondence between Harish-Chandra series in terms of unipotent series. 
 
 Let $\mathbf{L}$ be a rational Levi contained in a rational parabolic of a unitary group $\mathbf{G}$, and $L$ be its group of rational points, let $s$ be a semisimple element of $L$.  In view of Lemma \ref{LeviCentralizer}, it makes sense to consider inductions $R_{C_L(s)}^{C_G(s)}$, $R_{L_\#(s)}^{G_\#(s)}$ and $R_{L_{m-l}(s)}^{G_{m-l}(s)}$. The following result (similar to Proposition 8.25 in \cite{CE}) shows the effect of the Lusztig bijection on the parabolic induction from $L$ to $G$.

\begin{prop}\label{LusztigHarish-Chandra}
 Parabolic induction $R_L^G$ sends the series $\mathscr{E}(L,(s))$ to $\mathscr{R}(G,(s))$. Furthermore, the following diagram is commutative:
\begin{center}
\begin{tikzcd}
 \mathscr{E}(L,(s)) \arrow[d, "R_L^{G}"] \arrow[r, "\mathfrak{L}^L_s" swap, "\sim"]
 								&  \mathscr{E}(C_L(s),(1)) \arrow[d, "R_{C_L(s)}^{C_G(s)}"]  \\
 \mathscr{R}(G,(s)) \arrow[r, "\sim", "\mathfrak{L}^{G}_s" swap]
 								& \mathscr{R}(C_{G}(s),(1)).
 \end{tikzcd}
 \end{center}
\end{prop}  
\begin{proof}
Let $\pi$ belong to $\mathscr{E}(L,(s))$. Since $\mathbf{L}$ is a product of linear and unitary groups, central functions in $L$ are uniform. Therefore we can express $\pi$ as a linear combination $\pi=\sum_{s\in T} n_T R_T^L(s)$ of Deligne-Lusztig characters with integral coefficients. Transitivity of Lusztig induction implies that $R_L^G(\pi) = \sum_{s\in T}  n_T R_T^G(s)$, this representation belongs to $\mathscr{R}(G,(s))$.

Applying the Lusztig bijection we obtain 
$$
\mathfrak{L}^G_s \circ R_L^G(\pi)= \epsilon_\mathbf{G}\epsilon_{C_\mathbf{G}(s)}\sum_{s\in T} n_ T R_T^{C_G(s)}(s),
$$
this representation belongs to $\mathscr{R}(C_G(s),(1))$. On the other side, inducing the representation $\mathfrak{L}^L_s(\pi)= \epsilon_\mathbf{L}\epsilon_{C_\mathbf{L}(s)}\sum_{s\in T} n_T R_T^{C_L(s)}(s)$ to $C_G(s)$, we obtain
$$
R_{C_L(s)}^{C_G(s)} \circ \mathfrak{L}^L_s(\pi)= \epsilon_\mathbf{L}\epsilon_{C_{\mathbf{L}}(s)}\sum_{s\in T} n_T R_T^{C_G(s)}(s).
$$
By Lemma \ref{LeviCentralizer} the group $C_\mathbf{L}(s)$ is a rational Levi contained in a rational parabolic subgroup of $C_\mathbf{G}(s)$. Proposition \ref{RationalLeviFqRank} implies that $\epsilon_\mathbf{L}$ and $\epsilon_{C_\mathbf{L}(s)}$ are equal to $\epsilon_\mathbf{G}$ and $\epsilon_{C_\mathbf{G}(s)}$ respectively, whence the result. 
\end{proof}

\begin{cor}
The following diagram is commutative :
\begin{center}
 \begin{tikzcd}
 \mathscr{E}(L,(s)) \arrow[d, "R_L^{G_m}"] \arrow[r, "\Xi^{L}_s", "\sim" swap]
 								& \mathscr{E}(L_\#(s),(1))\times \mathscr{E}(L_{m-l},(1)) \arrow[d, "R_{L_\#(s)}^{G_\#(s)}\times R_{L_{m-l}}^{G_{m-l}}"] \\
 \mathscr{R}(G_m,(s)) \arrow[r, "\sim" swap, "\Xi^{G_m}_s"]  							
                          	& \mathscr{R}(G_\#(s),(1))\otimes \mathscr{R}(G_{m-l},(1)).
 \end{tikzcd}
\end{center}
\end{cor} 

Let $(G_m,G'_{m'})$ be a unitary dual pair. Theorem \ref{HoweCuspidalSupport} asserts that, for a cuspidal pair $(\mathbf{L},\rho)$ of $\mathbf{G}_m$, we can find a unique cuspidal pair $(\mathbf{L}',\rho')$ of $\mathbf{G}'_{m'}$ such that $\Theta^\flat_{m,m'}$ sends the series $\Irr(G_m,\rho)$ to $\mathscr{R}(G'_{m'},\rho')$.

\begin{prop}
Let $\omega^\flat_{m,m',\rho}$ denote the projection of the Weil representation $\omega^\flat_{m,m'}$ onto $\mathscr{R}(G_m,\rho)\otimes\mathscr{R}(G'_{m'},\rho')$. Then 
$$
\omega^\flat_{m,m'}=\bigoplus_{(\mathbf{L},\rho)}\omega^\flat_{m,m',\rho},
$$
where the sum runs over all rational conjugacy classes of cuspidal pairs of $G_m$.
\end{prop} 


Let the cuspidal pairs $(\mathbf{L},\rho)$ and $(\mathbf{L}',\rho')$ be as above, and let $s$ (resp. $s'$) be a semisimple element of $L$ (resp. $L'$) whose Lusztig series contains $\rho$ (resp. $\rho'$). 

\begin{prop}\label{DecompositionCuspidalPairViaXi}
The pairs $(L_\#(s),\rho_\#)$ and $(L'_\#(s'),\rho'_\#)$ are isomorphic. 
\end{prop}
\begin{proof}
According to Theorem \ref{HoweHarish-Chandra}, $L = \GL_\mathbf{t}\times T_r\times G_l$, $L' = \GL_\mathbf{t}\times T_{r'}\times G'_{l'}$, $\rho = \boldsymbol{\sigma}\otimes 1\otimes \varphi$ and $\rho' = \boldsymbol{\sigma}\otimes 1\otimes \varphi'$, where $\boldsymbol{\sigma}=\sigma_1\otimes\ldots\otimes\sigma_d$ is a product of non-trivial cuspidal representations and $\varphi'$ denotes the first occurrence of $\varphi$. 

Taking $s = s_{\boldsymbol{\sigma}}\times 1\times s_\varphi$ and $s' = s_{\boldsymbol{\sigma}}\times 1\times s_{\varphi'}$, where $s_\varphi$, $s_{\varphi'}$ and $s_{\boldsymbol{\sigma}}=s_1\times\cdots\times s_d$ are semisimple elements of $G_l$, $G'_{l'}$ and $GL_\mathbf{t}$ whose Lusztig series contain $\varphi$, $\varphi'$ and $\boldsymbol{\sigma}$, we have isomorphisms $L_\#(s) = (G_l)_\#(s_\varphi)\times (GL_\mathbf{t})_\#(s_{\boldsymbol{\sigma}})$, and $L' = (G'_{l'})_\#(s_{\varphi'})\times (GL_\mathbf{t})_\#(s_{\boldsymbol{\sigma}})$. Since $(G_l)_\#(s_\varphi)$ is isomorphic to $(G'_{l'})_\#(s_{\varphi'})$ (see Lemma \ref{PanGroups}), the groups $L_\#(s)$ and $L'_\#(s')$ are isomorphic. Theorem \ref{ReductionUnipotentCase} applied to the pair $(G_l,G_{l'})$ implies that $\varphi_\#\simeq\varphi'_\#$. This completes the proof.



\end{proof}

The preceding proposition allows us to identify the Harish-Chandra series $\Irr(G_\#(s),\rho_\#)$ and $\Irr(G'_\#(s'),\rho'_\#)$. Denote by $\mathbf{R}_{G_\#(s),\rho_\#}$ the projection of the representation $\mathbf{R}_{G_\#(s)}$ onto the series $\Irr(G_\#(s)\times G_\#(s),\rho_\#\otimes\rho_\#)$. The following is the main theorem of this section.

\begin{thm}\label{HC-UnipotentHC}
The representation $\omega^\flat_{m,m',\rho}$ is identified with $\mathbf{R}_{G_\#(s),\rho_\#}\otimes \omega^\flat_{m-l,m'-l,\rho_{m-l}}$  via the bijection
\begin{align}\label{HC-UnipotentHCBijection}
\Irr(G_m\times G'_{m'},\rho\otimes\rho')\simeq \Irr(C_{G_m}(s)\times C_{G'_{m'}}(s'),\rho_u\otimes \rho'_u),
\end{align}
where $s$ and $s'$ are rational semisimple elements of $L$ and $L'$ whose Lusztig series contain $\rho$ and $\rho'$ respectively.
\end{thm}
\begin{proof}
By Proposition \ref{LusztigHarish-Chandra}, the induced representation $R_{L\times L'}^{G_m\times G'_{m'}}(\rho\otimes\rho')$ belongs to $\mathscr{R}(G_m\times G'_{m'},(s\times s'))$. Therefore the Harish-Chandra series $\Irr(G_m\times G'_{m'},\rho\otimes\rho')$ is contained in $\mathscr{E}(G_m\times G'_{m'},(s \times s'))$.

Let $\pi$ and $\pi'$ be irreducible representations of $G_m$ and $G'_{m'}$ respectively. Proposition \ref{LusztigHarish-Chandra} implies that
$$
\langle \pi_u\otimes \pi'_u,  R_{C_L(s)}^{C_G(s)}(\rho_u)  \otimes R_{C_{L'}(s')}^{C_{G'}(s')} (\rho'_u)\rangle
=
\langle \pi,R_{C_L(s)}^{C_G(s)}\rho\rangle\langle\pi',R_{C_{L'}(s')}^{C_{G'}(s')} \rho'\rangle. 
$$
This means that the bijection Theorem \ref{ReductionUnipotentCase} restricts to (\ref{HC-UnipotentHCBijection}). The statement about the representation $\omega^\flat_{m,m',\rho}$ also follows. 
\end{proof}

From the proof we also realize that $\Theta^\flat_{m-l,m'-l}$ sends the Harish-Chandra series $\Irr(G_{m-l},\rho_{m-l})$ to $\mathscr{R}(G'_{m'-l},\rho'_{m'-l})$. As a consequence we have the following result. Its proof is the same as Corollary \ref{HoweLusztig} and will be omited.  
\begin{cor}\label{CorollaryHC-HCUnipotent}
 The following diagram is commutative :
\begin{center} 
 \begin{tikzcd}
  \Irr(G_m,\rho) \arrow[d, "\Theta^\flat_{m,m'}"] \arrow[r, "\Xi^{G_m}_s" swap, "\sim"] 
                                                & \Irr(G_\#(s),\rho_\#)\times \Irr(G_{m-l},\rho_{m-l}) \arrow[d, "\iota\hspace{0.05cm}\otimes\hspace{0.05cm}\Theta^\flat_{m-l,m'-l}"] \\
\mathscr{R}(G'_{m'},\rho') \arrow[r, "\sim", "\Xi^{G'_{m'}}_{s'}" swap] 
 												   & \mathscr{R}(G_\#(s),\rho_\#)\otimes \mathscr{R}(G'_{m'-l},\rho'_{m'-l}).
 \end{tikzcd}
\end{center}
\end{cor}
\section{Extremal representations}\label{SectionExtremalRepresentations}
In previous section we saw how the Lusztig correspondence allows the study of the Howe correspondence to be brought to the study of its effect on unipotent representations. This enables us to extend results in \cite{Epequin}. 
 
Let $(G_m,G'_{m'})$ denote an unitary dual pair, and $\pi$ be an irreducible representation of $G_m$. We start by using Corollary \ref{CorollaryHC-HCUnipotent} to reduce the set $\Theta^\flat_{m,m'}(\pi)$ to a set of unipotent representations.

\begin{prop}\label{bijectionThetaThetaUnipotent}
The representation $\pi'$ of $G'_{m'}$ belongs to $\Theta^\flat_{m,m'}(\pi)$, if and only if, $\pi'_\#=\hat{\pi}_\#$ and $\pi'_{m'-l}$ belongs to $\Theta^\flat_{m-l,m'-l}(\pi_{m-l})$. In particular, the map sending $\pi'\in\Irr(G'_{m'})$ to $\pi'_{m'-l}\in\Irr(G'_{m'-l})$ defines a bijection between $\Theta^\flat_{m,m'}(\pi)$ and $\Theta^\flat_{m-l,m'-l}(\pi_{m-l})$.
\end{prop}
This proposition tells us, in other words, that the Lusztig correpondence restricts to 
$$
\Xi_s^{G'_{m'}}:\Theta^\flat_{m,m'}(\pi)\simeq \{\hat{\pi}_\#\}\times\Theta^\flat_{m-l,m'-l}(\pi_{m-l}).
$$

Definition 7 in \cite{Epequin} introduced a partial order on the set $\Theta^\flat_{m,m'}(\pi)$, where $\pi$ is a unipotent representation of $G_m$. Using Proposition \ref{bijectionThetaThetaUnipotent}, we can extend this order to arbitrary irreducible representations. 

\begin{defi}
Let $\pi'$ and $\sigma'$ belong to $\Theta^\flat_{m,m'}(\pi)$. Then $\pi'\leq\sigma'$ if and only if $\pi'_{m'-l}\leq\sigma'_{m'-l}$. This defines a partial order in $\Theta^\flat_{m,m'}(\pi)$.
\end{defi}
When $\pi$ is a unipotent representation, the order defined on $\Theta^\flat_{m,m'}(\pi)$ is not necessarily total. However, in Theorems 9 and 10 of \cite{Epequin} we were able to find a minimal and a maximal representations for this order. The same is true for arbitrary irreducible representations.

\begin{thm}\label{ExtremalRepresentations}
Let $\pi$ be an irreducible representation of $G_m$. There exists a unique minimal (resp. maximal) irreducible representation $\pi'_\mathrm{min}$ (resp. $\pi'_\mathrm{max}$) in $\Theta^\flat_{m,m'}(\pi)$.
\end{thm}
\begin{proof}
Let $\sigma$ denote the unipotent representation $\pi_{m-l}$. According to Theorems 9 and 10 in \cite{Epequin}, there is an irreducible representation $\sigma'_\mathrm{min}$ (resp. $\sigma'_\mathrm{max}$) in $\Theta^\flat_{m-l,m'-l}(\sigma)$ verifying $\sigma'_\mathrm{min}\leq \sigma'$ (resp. $\sigma'\leq \sigma'_\mathrm{max}$) for all $\sigma'$ in $\Theta^\flat_{m-l,m'-l}(\sigma)$. Thanks to Proposition \ref{bijectionThetaThetaUnipotent}, we see that $\pi'_\mathrm{min}$ (resp. $\pi'_\mathrm{max}$) verifying $(\pi'_\mathrm{min})_{m'-l}=\sigma'_\mathrm{min}$ (resp. $(\pi'_\mathrm{max})_{m'-l}=\sigma'_\mathrm{max}$) is the desired minimal (resp. maximal) representation. 
\end{proof}
\section{Correspondence for Weyl groups}
Let $\mathbf{G}$ be a reductive group defined over $\mathbb{F}_q$, and $\mathbf{P}=\mathbf{L}\mathbf{U}$ be a Levi decomposition of the rational parabolic subgroup $\mathbf{P}$. For a cuspidal representation $\delta$ of $L$ set
 $$
 W_\mathbf{G}(\delta) = \{x\in N_{G}(\mathbf{L})/L : {}^x\delta=\delta\}.
 $$
\begin{prop}\cite[Corollary 5.4]{HL} and \cite[Corollary 2]{Geck}\label{HLBijection}
There is an isomorphism 
$$
\End_G(R_L^G(\delta)) \simeq \mathbb{C}[W_\mathbf{G}(\delta)].
$$
In particular, irreducible representations in the \emph{Harish-Chandra series} $\Irr(\mathbf{G},\delta)$ are indexed by irreducible representations of $W_\mathbf{G}(\delta)$. We refer to this as the \emph{Howlett-Lehrer bijection}.
\end{prop}

Between the classical groups, those having cuspidal unipotent representations are scarce. For instance, the only linear group affording a unipotent cuspidal representation is $\GL_1(q)$, the representation being trivial. The only unitary groups having such a representation are $U_{k(k+1)/2}(q)$ for non-negative integers $k$. Furthermore, this representation is unique, and we denote it by $\lambda_k$. Since cuspidal representations are preserved on the first occurrence, the Howe correspondance between cuspidal unipotent representations can be seen as a correspondance (also denoted by) $\theta^\flat$ between non-negative integers (cf. Section 2 in \cite{Epequin}).

Let $m(k)$ denote the Witt index of the group $U_{k(k+1)/2}(q)$. The discussion above implies that unipotent Harish-Chandra series of the unitary group $G_m$ can be parametrized by cuspidal pairs $(G_{m(k)}\times T_{m-m(k)},\lambda_k\otimes 1)$, for non-negative integers $k$. The corresponding series will be denoted by $\Irr(G_m)_k$. The unicity of $\lambda_k$ implies that the condition on the elements of the group $W_{G_m}(\lambda_k \otimes 1)$ is trivial. Therefore, the Howlett-Lehrer bijection identifies $\Irr(G_m)_k$ to the set of irreducible representation of $N_{G_m}(L_k)/L_k$, which is a Weyl group of type $\mathbf{B}_{m-m(k)}$.

Let $(G_m,G'_{m'})$ be a unitary dual pair. Let $(\mathbf{L},\rho)$ be a cuspidal pair of $\mathbf{G}_m$, and $(\mathbf{L}',\rho')$ be the unique cuspidal pair of $\mathbf{G}'_{m'}$ such that $\Theta^\flat_{m,m'}$ sends the series $\Irr(G_m,\rho)$ to $\mathscr{R}(G'_{m'},\rho')$.
Proposition \ref{HLBijection} yields a bijection 
$$
\Irr(W_{\mathbf{G}_m}(\rho)\times W_{\mathbf{G}'_{m'}}(\rho'))\simeq \Irr(G_m\times G'_{m'},\rho\otimes\rho').
$$
This bijection identifies the representation $\omega^\flat_{m,m',\rho}$ with a representation of the direct product of Weyl groups $W_{\mathbf{G}_m}(\rho)\times W_{\mathbf{G}'_{m'}}(\rho')$ that we denote by $\Omega_{m,m',\rho}$. When $\rho$ is the cuspidal unipotent representation $\lambda_k\otimes 1$, Corollary \ref{HoweUnipotentHC} implies that $\rho'$ is equal to $\lambda_{k'}\otimes 1$, where $k'=\theta^\flat(k)$. In this case $\Omega_{m,m',\rho}$ is denoted by $\Omega_{m,m',k}$. 

Let $r$ (resp. $r'$) be equal to $m-m(k)$ (resp. $m'-m'(k')$). Aubert, Michel and Rouquier proved the following. 

\begin{thm}\cite[Theorem 3.10]{AMR}\label{ReductionWeilUnipotent}
Let $(G_m,G'_{m'})$ be a unitary dual pair. The bijection
$$
\Irr(G_m)_k\times \Irr(G'_{m'})_{k'}\simeq \Irr(W_{m-m(k)}\times W_{m'-m'(k')}),
$$
identifies $\omega^\flat_{m,m',k}$ with the representation $\Omega_{m,m',k}$ whose character is :
\begin{align}\label{u1}
\sum_{l=0}^{\min(r,r')}\sum_{\chi\in\Irr(W_l)}(\Ind_{W_l\times W_{r-l}}^{W_r}\chi \otimes 1) \otimes (\Ind_{W_l\times W_{r'-l}}^{W_r'}\sgn\chi \otimes 1),
\end{align}
when $k$ is odd or $k=k'=0$; and
\begin{align}\label{u2}
\sum_{l=0}^{\min(r,r')}\sum_{\chi\in\Irr(W_l)}(\Ind_{W_l\times W_{r-l}}^{W_r}\chi \otimes \sgn) \otimes (\Ind_{W_l\times W_{r'-l}}^{W_r'}\sgn\chi \otimes 1),
\end{align}
otherwise.
\end{thm}  

Let $(\mathbf{L},\rho)$ and $(\mathbf{L}',\rho')$ be as above. Since the unipotent representation $\rho_{m-l}$ of $G_{m-l}$ is also cuspidal, there exist a non-negative integer $k$ such that the cuspidal pair $(L_{m-l},\rho_{m-l})$ is equal to $(G_{m(k)}\times T_r,\lambda_k\otimes 1)$, where $r$ is equal to $m-l-m(k)$. Let $k'$ be equal to $\theta^\flat(k)$, and $r'$ to $m'-l-m(k')$. 

The following is the main result of this section, it expresses $\Omega_{m,m',\rho}$ in terms of the representation described in the previous theorem. 

\begin{thm}\label{GeneralizedWeylGroupCorrespondence}
The Weyl group $W_{\mathbf{G}_m}(\rho)$ (resp. $W_{\mathbf{G}'_{m'}}(\rho')$) is isomorphic to $W_{\mathbf{G}_\#(s)}(\rho_\#)\times W_r$ (resp. $ W_{\mathbf{G}_\#(s)}(\rho_\#)\times W_{r'}$). The ensuing bijection
$$
\Irr(W_{\mathbf{G}_m}(\rho)\times W_{\mathbf{G}'_{m'}}(\rho'))\simeq \Irr(W_{\mathbf{G}_\#(s)}(\rho_\#)\times W_{\mathbf{G}_\#(s)}(\rho_\#))\times \Irr(W_r,W_{r'}),
$$
identifies the representation $\Omega_{m,m',\rho}$ with $\mathbf{R}_{W_{\mathbf{G}_\#(s)}(\rho_\#)}\otimes \Omega_{m-l,m'-l,k}$.   
\end{thm}
\begin{proof}
Theorem \ref{HC-UnipotentHC} states that $\omega^\flat_{m,m',\rho}$ is identified with $\mathbf{R}_{G_\#(s),\rho_\#}\otimes \omega^\flat_{m-l,m'-l,k}$ via the Lusztig correspondence
$$
\Irr(G_m\times G'_{m'},\rho\otimes\rho')\simeq \Irr(C_{G_m}(s_m)\times C_{G'_{m'}}(s'_{m'}),\rho_u\otimes \rho'_u).
$$

Proposition \ref{HLBijection} furnishes a bijection
$$
\Irr(G_\#(s)\times G_\#(s),\rho_\#\otimes\rho_\#)\simeq\Irr(W_{\mathbf{G}_\#(s)}(\rho_\#)\times W_{\mathbf{G}_\#(s)}(\rho_\#)),
$$
which clearly identifies $\mathbf{R}_{G_\#(s),\rho_\#}$ with $\mathbf{R}_{W_{\mathbf{G}_\#(s)}(\rho_\#)}$. 

Corollary \ref{HoweUnipotentHC} implies that the cuspidal representation $\rho'_{m'-l}$ is equal to $\lambda_{k'}\otimes 1$, where $r'=m'-l-m(k')$. In that case, Theorem \ref{ReductionWeilUnipotent} implies that the bijection 
$$
\Irr(G_{m-l})_k\times\Irr(G'_{m'-l})_{k'}\simeq \Irr(W_r,W_{r'}),
$$
identifies $\omega^\flat_{m-l,m'-l,k}$ with the representation $\Omega_{m-l,m'-l,k}$.  The bijection in the statement comes from composing the last three bijections. The isomorphisms between the Weyl groups come from Chapter 28 in \cite{Lusztig2}.
\end{proof}
\bibliography{Paper2bib}

\begin{thebibliography}{10}

\bibitem{AM}
{\sc Adams, J., and Moy, A.}
\newblock Unipotent representations and reductive dual pairs over finite
  fields.
\newblock {\em Transactions of the american mathematical society 340}, 1
  (1993), 309--321.

\bibitem{AMR}
{\sc Aubert, A.-M., Michel, J., and Rouquier, R.}
\newblock Correspondance de {H}owe pour les groupes r\'eductifs sur les corps
  finis.
\newblock {\em Duke Math. J. 83}, 2 (1996), 353--397.

\bibitem{CE}
{\sc Cabanes, M., and Enguehard, M.}
\newblock {\em Representation Theory of Finite Reductive Groups}.
\newblock New Mathematical Monographs. Cambridge University Press, 2004.

\bibitem{Digne-Michel}
{\sc Digne, F., and Michel, J.}
\newblock {\em Representations of finite groups of Lie type}.
\newblock London Mathematical Society Student Texts (Book 21). Cambridge
  University Press, 1991.

\bibitem{Epequin}
{\sc Epequin~Chavez, J.}
\newblock Extremal unipotent representations for the finite howe
  correspondence.
\newblock {\em Journal of Algebra 535\/} (2019), 480 -- 502.

\bibitem{Geck}
{\sc Geck, M.}
\newblock A note on {H}arish-{C}handra induction.
\newblock {\em Manuscripta mathematica 80}, 4 (1993), 393--402.

\bibitem{Gerardin}
{\sc G\'erardin, P.}
\newblock Weil representations associated to finite fields.
\newblock {\em J. Algebra 46}, 1 (1977), 54--101.

\bibitem{Howe}
{\sc Howe, R.}
\newblock Transcending classical invariant theory.
\newblock {\em Journal of the American Mathematical Society 2}, 3 (1989),
  535--552.

\bibitem{HL}
{\sc Howlett, R., and Lehrer, G.}
\newblock Induced cuspidal representations and generalised hecke rings.
\newblock {\em Inventiones mathematicae 58\/} (1980), 37--64.

\bibitem{Kudla}
{\sc Kudla, S.~S.}
\newblock On the local theta-correspondence.
\newblock {\em Inventiones mathematicae 83\/} (1986), 229--256.

\bibitem{Lusztig3}
{\sc Lusztig, G.}
\newblock Irreducible representations of finite classical groups.
\newblock {\em Inventiones mathematicae 43}, 2 (Jun 1977), 125--175.

\bibitem{Lusztig2}
{\sc Lusztig, G.}
\newblock {\em Characters of Reductive Groups over a Finite Field. (AM-107)}.
\newblock Princeton University Press, 1984.

\bibitem{Lusztig}
{\sc Lusztig, G.}
\newblock On the representations of reductive groups with disconnected centre.
\newblock In {\em Orbites unipotentes et repr\'esentations - I. Groupes finis
  et Alg\`ebres de Hecke}, no.~168 in Ast\'erisque. Soci\'et\'e math\'ematique
  de France, 1988, pp.~157--166.

\bibitem{VMW}
{\sc M{\oe}glin, C., Vign\'eras, M.-F., and Waldspurger, J.-L.}
\newblock {\em Correspondances de {H}owe sur un corps {$p$}-adique}, vol.~1291
  of {\em Lecture Notes in Mathematics}.
\newblock Springer-Verlag, Berlin, 1987.

\bibitem{Pan}
{\sc Pan, S.-Y.}
\newblock Supercuspidal representations and preservation principle of theta
  correspondence.
\newblock {\em Journal für die reine und angewandte Mathematik (Crelles
  Journal)\/} (2016).

\bibitem{Pan2}
{\sc Pan, S.-Y.}
\newblock Howe correspondence of unipotent characters for a finite
  symplectic/even-orthogonal dual pair.
\newblock {\em arXiv preprint arXiv:1901.00623\/} (2019).

\end{thebibliography}
\bibliographystyle{acm}
\end{document}